\documentclass[a4paper, 11pt]{article}
\usepackage{amsmath,amssymb,amsthm,xypic} 
\input xy
\xyoption{all}

\usepackage[hypertex]{hyperref}

\usepackage[totalheight=24 true cm, totalwidth=15 true cm]{geometry}
\usepackage{color}

\title{Skolem-Mahler-Lech type theorems and Picard-Vessiot theory}
\author{Michael Wibmer}

\newtheorem{theo}{Theorem}[section]
\newtheorem{lemma}[theo]{Lemma}
\newtheorem{prop}[theo]{Proposition}
\newtheorem{cor}[theo]{Corollary}
\newtheorem{defi}[theo]{Definition}
\newtheorem{rem}[theo]{Remark}

\theoremstyle{definition}
\newtheorem{ex}[theo]{Example}

\newcommand{\oo}{\mathcal{O}}

\newcommand{\ida}{\mathfrak{a}}

\newcommand{\Gl}{\operatorname{GL}}

\newcommand{\V}{\mathbb{V}}

\newcommand{\A}{\mathbb{A}}

\newcommand{\s}{\sigma}

\newcommand{\seq}{\operatorname{Seq}}

\newcommand{\N}{\mathbb{N}}

\begin{document}

\maketitle

\begin{abstract}
We show that three problems involving linear difference equations with rational function coefficients are essentially equivalent. The first problem is the generalization of the classical Skolem-Mahler-Lech theorem to rational function coefficients. The second problem is the question whether or not for a given linear difference equation there exists a Picard-Vessiot extension inside the ring of sequences. The third problem is a certain special case of the dynamical Mordell-Lang conjecture.

This allows us to deduce solutions to the first two problems in a particular but fairly general special case.
\end{abstract}
 \footnotetext{e-mail: michael.wibmer@matha.rwth-aachen.de}

\section{Introduction}

Let $k$ be a field of characteristic zero and \[E=\s^n(y)+h_{n-1}\s^{n-1}(y)+\cdots+h_0y\] a linear difference equation with rational function coefficients $h_{n-1},\ldots,h_0\in k(z)$. We shall be concerned with the following two problems:

\bigskip

\noindent {\bf Problem SML:}
Let $f\colon \N\to k$ be a sequential solution of $E$, that is
\[f(i+n)+h_{n-1}(i)f(i+n-1)+\cdots+h_0(i)f(i)=0 \text{ for } i\gg 0.\]
Is it true that the set $\{i\in \N|\ f(i)=0\}$ is a finite union of arithmetic progressions?

\bigskip

\noindent {\bf Problem PV:}
Does there exist a Picard-Vessiot extension of $k(z)$ for $E=0$ inside the ring of $k$-valued sequences?

\bigskip

As we shall see shortly these two problems are intimately connected.
If the difference equation $E$ has constant coefficients, i.e., if $h_{n-1},\ldots,h_0\in k$, then the answer to problem SML is affirmative. This result is known as the Skolem-Mahler-Lech theorem. The reader is referred to the introduction of \cite{Bell:AGeneralizedSkolemMahlerLechTheorem} for proper attributions and more background on this celebrated theorem. Using $p$-adic techniques it has recently been shown in \cite{Belletal:OnTheSetOfZerocoefficientsofafunctionSatisfying} that the answer is also affirmative if $h_{n-1},\ldots,h_1\in k[z]$ are polynomials and $h_0\in k\smallsetminus\{0\}$ is a non-zero constant. Here we will give a new proof of this result based on Picard-Vessiot theory and a special case of the dynamical Mordell-Lang conjecture established by J. Bell in \cite{Bell:AGeneralizedSkolemMahlerLechTheorem}. We also explain how a certain special case of the dynamical Mordell-Lang conjecture would imply an affirmative solution to problems SML and PV in general. We solve Problem PV affirmatively under the restriction $h_{n-1},\ldots,h_1\in k[z]$ and $h_0\in k\smallsetminus\{0\}$.

\bigskip

{\bf Acknowledgement:} I would like to thank Jason Bell, Ruyong Feng and Julia Hartmann for helpful comments.

\section{Picard-Vessiot theory and preliminaries} \label{sec:PV}

We start by recalling the basic definitions from Picard-Vessiot theory, the Galois theory of linear difference equations. The standard reference for this is \cite{SingerPut:difference}.
We will also use some standard notations from difference algebra, see \cite{Cohn:difference} or \cite{Levin:difference}.

\bigskip

All rings are assumed to be commutative. We set $\N:=\{0,1,2\ldots\}$. A difference ring, or $\s$-ring for short, is a ring $R$ equipped with a ring endomorphism $\s\colon R\to R$. A morphism of difference rings is a morphism of rings that commutes with $\s$. A $\s$-field is a $\s$-ring whose underlying ring is a field. The difference field we shall be most interested in is the field $K=k(z)$ of rational functions in the variable $z$ over a field $k$ of characteristic zero, equipped with the automorphism $\s$ defined by
$\s(h(z))=h(z+1)$ for $h\in k(z)$. For a difference ring $R$ the constants of $R$ are $R^\s:=\{r\in R|\ \s(r)=r\}$. Note that in our example, $K^\s=k$.

\medskip

Instead of a linear difference equation of order $n$, as in the introduction, we shall consider the slightly more general situation of a square first order system
$\s(y)=Ay$ where $A\in\Gl_n(K)$ and $y$ is a vector of length $n$.

\begin{defi}
Let $K$ be a $\s$-field and $A\in\Gl_n(K)$. A $K$-$\s$-algebra is called a \emph{Picard-Vessiot ring} for $\s(y)=Ay$ if
\begin{enumerate}
 \item $R$ is generated by a fundamental solution matrix for $\s(y)=Ay$, i.e., there exists $Y\in\Gl_n(R)$
 such that $\s(Y)=AY$ and $R=K[Y_{ij},\tfrac{1}{\det(Y)}]$.
\item $R$ is $\s$-simple, i.e., $R$ has no non-trivial $\s$-ideals.
\end{enumerate}
\end{defi}

It is a characteristic feature of the Galois theory of linear difference equations that the Picard-Vessiot ring is usually not an integral domain, contrary to the case of linear differential equations (see e.g. \cite{SingerPut:differential}.) However, like any Noetherian $\s$-simple $\s$-ring $R$, a Picard-Vessiot ring has a certain simple algebraic structure (see \cite[Cor. 1.16, p. 12]{SingerPut:difference} or \cite[Prop. 1.1.2, p. 2]{Wibmer:thesis}). Namely, there exist orthogonal idempotents $e_0,\ldots,e_{l-1}\in R$ such that
\begin{enumerate}
\item $R=e_0R\oplus\cdots\oplus e_{l-1}R$,
\item $\s(e_0)=e_1,\s(e_1)=e_2,\ldots,\s(e_{l-1})=e_0$,
\item the ring $e_iR$ is $\s^l$-simple and an integral domain for $i=0,\ldots,l-1$.
\end{enumerate}
We will call the integer $l$ the \emph{period} of $R$. It is precisely this simple algebraic structure of the Picard-Vessiot ring which is at the core of the connection between Picard-Vessiot theory and Problem SML. In fact we will see (Proposition \ref{prop:period}) that in many cases the period of the Picard-Vessiot ring is a bound for the period of the arithmetic progressions appearing in Problem SML.

It is sometimes more convenient or even necessary (e.g. for the Galois correspondence) to work with the total ring of fractions of the Picard-Vessiot ring rather than with the Picard-Vessiot ring itself. If $L$ is the total ring of fractions of a Noetherian $\s$-simple $\s$-ring then there exist orthogonal idempotents $e_0,\ldots,e_{l-1}\in L$ such that
\begin{enumerate}
\item $L=e_0L\oplus\cdots\oplus e_{l-1}L$,
\item $\s(e_0)=e_1,\s(e_1)=e_2,\ldots,\s(e_{l-1})=e_0$,
\item $e_iL$ is a field for $i=0,\ldots,l-1$.
\end{enumerate}

Conversely, if $L$ is a $\s$-ring satisfying the above three properties then $L$ is a Noetherian $\s$-simple $\s$-ring such that every non-zero divisor is invertible. We will call such $\s$-rings \emph{$\s$-pseudo fields}.

%

\begin{defi} \label{defi:PVextension}
Let $K$ be a $\s$-field and $A\in\Gl_n(K)$. A $\s$-pseudo field extension $L$ of $K$ is called a \emph{Picard-Vessiot extension} for $\s(y)=Ay$, if
\begin{enumerate}
\item $L$ is generated by a fundamental solution matrix for $\s(y)=Ay$, i.e., there exists $Y\in\Gl_n(L)$ such that $\s(Y)=AY$ and $L=K(Y_{ij})$. (To be precise, this notation means that $L$ is the total ring of fractions of $K[Y_{ij}]$.)
\item  $L^\s=K^\s$.
\end{enumerate}
\end{defi}

It is easy to see that a Picard-Vessiot ring always exists for a given equation $\s(y)=Ay,\ A\in\Gl_n(K)$:
Consider the universal solution ring $U=K[Z_{ij},\frac{1}{\det{Z}}]$ where $Z$ is an $n\times n$-matrix of indeterminates. We consider $U$ as $K$-$\s$-algebra by setting $\s(Z)=AZ$.
Then any quotient of $U$ modulo a maximal element in the set of all $\s$-ideals of $U$ is a Picard-Vessiot ring.

Under the assumption that $\s\colon K\to K$ is surjective and $K^\s$ is algebraically closed it is shown in \cite{SingerPut:difference} that a Picard-Vessiot ring for a given linear system $\s(y)=Ay$ is unique up to $K$-$\s$-isomorphisms. Moreover, the total ring of fractions of a Picard-Vessiot ring is a Picard-Vessiot extension. Under the additional hypothesis that $K$ is perfect it is also shown that $K[Y_{ij},\frac{1}{\det(Y)}]\subset L$ is a Picard-Vessiot ring for $\s(y)=Ay$ if $L$ is a Picard-Vessiot extension for $\s(y)=Ay$ with fundamental solution matrix $Y\in \Gl_n(L)$.

If $K^\s$ is not algebraically closed, a Picard-Vessiot extension for $\s(y)=Ay$ might or might not exist. However, if $K=k(z)$ with $\s(z)=z+1$ then one can show that there always exists a Picard-Vessiot extension (Proposition \ref{prop:existence}).

\bigskip

Let $K$ be a $\s$-field and $R$ a $K$-$\s$-algebra. An element $f\in R$ is called $\s$-finite (over $K$) if $f$ satisfies a linear $\s$-equation over $K$. This is equivalent to saying that $f$ is contained in a finite dimensional $K$-subvector space of $R$ which is stable under $\s$. For later reference we note the following simple fact:
\begin{lemma}\label{lemma:sfinite}
Let $R$ be a Picard-Vessiot ring for some equation $\s(y)=Ay$ over some $\s$-field $K$. Then every element of $R$ is $\s$-finite over $K$.
\end{lemma}
\begin{proof}
Let $Y\in\Gl_n(R)$ be a fundamental solution matrix for $\s(y)=Ay$. Then the finite dimensional $K$-subvector space of $R$ generated by all $Y_{ij}$ and $\det(Y)$ is stable under $\s$. The claim follows because sums and products of $\s$-finite elements are $\s$-finite.
\end{proof}

\bigskip

\bigskip

Let $k$ be a field of characteristic zero. We define the ring $\seq_k$ of $k$-valued sequences as in \cite[Example 1.3, p.4]{SingerPut:difference}: This means that two sequences $f,g\colon \N\to k$ are identified if $f(i)=g(i)$ for all $i\gg 0$. Addition and multiplication is componentwise. We shall be rather careless about the distinction between an element $f\in k^\N$ and its equivalence class in $\seq_k$. In particular, the equivalence class of $f$ will often be denoted by $(f(0),f(1),\ldots)$.
We consider $\seq_k$ as $\s$-ring by shifting to the left, that is
\[\s(f)=(f(1),f(2),\ldots).\]

For $f\in\seq_k$ the set $\{i\in \N|\ f(i)=0\}$ is strictly speaking not well-defined. However, if we identify two subsets $M_1, M_2$ of $\N$ if the difference $M_1\smallsetminus M_2$ (or equivalently $M_2\smallsetminus M_1$) is finite, then $\{i\in \N|\ f(i)=0\}$ yields a well defined equivalence class. In the sequel we will usually consider subsets of $\N$ modulo this equivalence relation.

By an \emph{arithmetic progression} we mean a subset of $\N$ of the form $j+\N l$ with $j,l\in\N$. In particular, the one element set $\{j\}$ is considered to be an arithmetic progression. Therefore the property of a subset of $\N$ to be a finite union of arithmetic progressions passes down to equivalence classes. The integer $l$ is called the \emph{period} of the arithmetic progression $j+\N l$. 

One advantage of the ring $\seq_k$ compared to the ring $k^\N$ is that we can consider $k(z)$ as a subring of $\seq_k$. In fact, the map
\[k(z)\to\seq_k,\ h\mapsto (h(0),h(1),h(2),\ldots)\] is a morphism of difference rings, i.e., $\seq_k$ is a $k(z)$-$\s$-algebra. (The expression $(h(0),h(1),h(2),\ldots)$ is well-defined because a rational function has only finitely many poles.)

For a given linear system $\s(y)=Ay$ with $A\in\Gl_n(K)$ where $K=k(z),\ \s(z)=z+1$ it seems quite natural to ask whether or not there exists a Picard-Vessiot extension for $\s(y)=Ay$ inside $\seq_k$. The fact that $(\seq_k)^\s=k=K^\s$ seems to speak in favor of an affirmative answer. In \cite[Prop. 4.1, p. 45]{SingerPut:difference} it is shown, under the assumption that $k$ is algebraically closed, that there always exists a Picard-Vessiot ring inside $\seq_k$. This is not sufficient to deduce the existence of a Picard-Vessiot extension inside $\seq_k$ because a non-zero divisor in the Picard-Vessiot ring could be a zero divisor in $\seq_k$.

The following proposition removes the restriction of $k$ being algebraically closed.

\begin{prop} \label{prop:existence}
Let $k$ be a field of characteristic zero and $A\in\Gl_n(k(z))$. Then there exists a Picard-Vessiot ring for $\s(y)=Ay$ inside $\seq_k$. Moreover there exists a Picard-Vessiot extension for $\s(y)=Ay$.
\end{prop}
\begin{proof}
It is easy to construct a fundamental solution matrix $Y$ for $\s(y)=Ay$ inside $\seq_k$: If we choose $i_0\in \N$ big enough then $A(i)$ is defined and $\det(A)(i)$ is non-zero for every $i\geq i_0$. Set $Y(i_0)$ to be the identity matrix and recursively define $Y(i+1)=A(i)Y(i)$ for $i\geq i_0$. Then $Y\in\Gl_n(\seq_k)$ and $\s(Y)=AY$.
We have to show that $R:=k(z)[Y_{ij},\tfrac{1}{\det(Y)}]\subset\seq_k$ is $\s$-simple.

Let $\overline{k}$ denote an algebraic closure of $k$. We know from \cite[Prop. 4.1, p. 45]{SingerPut:difference} that $\overline{R}:=\overline{k}(z)[Y_{ij},\tfrac{1}{\det(Y)}]$ is a Picard-Vessiot ring for $\s(y)=Ay$ over $\overline{k}(z)$. Since the natural map $\seq_k\otimes_k \overline{k}\to \seq_{\overline{k}}$ is injective also $R\otimes_k\overline{k}\to \overline{R}$ is injective. Therefore $\overline{R}=R\otimes_k\overline{k}$. A non-trivial $\s$-ideal $\ida$ of $R$ would give rise to a non-trivial $\s$-ideal $\ida\otimes_k\overline{k}$ of $\overline{R}$. Consequently $R$ must be $\s$-simple, i.e., $R$ is a Picard-Vessiot ring for $\s(y)=Ay$.

The total quotient ring $L$ of $R$ is a $\s$-pseudo field, generated by a fundamental solution matrix for $\s(y)=Ay$. It remains to see that $L^\s=k$. So let $g\in L^\s$. The set
\mbox{$\{f\in R|\ fg\in R\}$} is a non-zero $\s$-ideal of $R$. Because $R$ is $\s$-simple this ideal must contain $1$, i.e., $g\in R^\s$. But $R^\s=k$ because $(\seq_k)^\s=k$.
\end{proof}

A Picard-Vessiot ring or a Picard-Vessiot extension for a linear equation $\s(y)=Ay$ over $K=k(z)$ need not be unique up to $K$-$\s$-isomorphisms if $k=K^\s$ is not algebraically closed. However, inside $\seq_k$ there only exists one Picard-Vessiot ring:

\begin{rem} \label{rem:PVringunique}
Let $k$ be a field of characteristic zero and $A\in\Gl_n(k(z))$. The Picard-Vessiot ring $R\subset\seq_k$ for $\s(y)=Ay$ is set theoretically unique. In fact, if $Y\in\Gl_n(\seq_k)$ is any fundamental solution matrix for $\s(y)=Ay$ then $R=k(z)[Y_{ij},\tfrac{1}{\det(Y)}]$.
\end{rem}
\begin{proof}
Let $Y'\in\Gl_n(R)\subset\Gl_n(\seq_k)$ be a fundamental solution matrix for $\s(y)=Ay$. Then $Y'=YC$ for some matrix $C\in\Gl_n((\seq_k)^\s)=\Gl_n(k)$. (Simply compute $\s(Y^{-1}Y')$.)
Therefore $R=k(z)[Y'_{ij},\tfrac{1}{\det(Y')}]=k(z)[Y_{ij},\tfrac{1}{\det(Y)}]$.
\end{proof}

\section{Dynamical Mordell-Lang and main results}

Let $\s\colon X\dashrightarrow X$ be a rational map on a variety $X$. Let $U$ be the largest open subset such that $\s$ is defined on $U$. For $x\in X$ we say that the orbit $\oo_\s(x)$ is defined if $x\in U,\s(x)\in U,\ldots$. The dynamical Mordell-Lang conjecture is a fundamental problem in algebraic dynamics:

\bigskip

\noindent {\bf Problem DML:} Let $X$ be a variety over a field of characteristic zero equipped with a rational map $\s\colon X\dashrightarrow X$. Let $x\in X$ be such that $\oo_\s(x)$ is defined and let $Y$ be a closed subvariety of $X$.

When is it true that the set $\{i\in\N|\ \s^i(x)\in Y\}$ is a finite union of arithmetic progressions?

\bigskip

Several special cases of the dynamical Mordell-Lang conjecture have already been established. We refer the reader to \cite{BellGhiocaTucker:DynamicalMordellLangProblemforEtaleMaps} and \cite{GhiocaTucker:PeriodicPointsLinearizingMapsandtheDynamicalMordellLangProblem} for an overview. The most relevant special case for us is due to J. Bell \cite{Bell:AGeneralizedSkolemMahlerLechTheorem} and solves the case when $X$ is affine and $\s\colon X\to X$ is an everywhere defined invertible morphism.

Now we are prepared to prove our main result:
\begin{theo} \label{theo:main}
Let $k$ be a field of characteristic zero and $A\in\Gl_n(k(z))$. Set $X=\A^1_k\times\Gl_{n,k}$ and define a rational map $\s\colon X\dashrightarrow X$ by \[\s(b,B)=(b+1, A(b)B).\]
Then the following statements are equivalent:
\begin{enumerate}
\item For every $x=(b,B)\in X(k)$ with $b\in\N$ and $\oo_\s(x)$ defined and every closed subvariety $Y\subset X$, the set $\{i\in\N|\  \s^i(x)\in Y\}$ is a finite union of arithmetic progressions.
\item There exists a point $x=(b,B)\in X(k)$ with $b\in\N$ and $\oo_\s(x)$ defined such that for every closed subvariety $Y\subset X$ the set $\{i\in\N|\  \s^i(x)\in Y\}$ is a finite union of arithmetic progressions.
\item There exists a Picard-Vessiot extension of $k(z)$ for $\s(y)=Ay$ inside $\seq_k$.
\end{enumerate}
\end{theo}
\begin{proof}
We start by showing that (iii) implies $(\rm{i})$:  Let $x=(b,B)\in X(k)$ be such that $b\in\N$ and $\oo_\s(x)$ is defined. We can define a map $\psi$ from $k[X]=k[z,Z_{ij},\tfrac{1}{\det(Z)}]$ to $\seq_k$ by setting
\[\psi(f):=\left(*,\ldots,*,f(x),f(\s(x)),f(\s^2(x)),\ldots\right)\]
where $f(x)$ is on position $b\in\N$. Then $\psi$ is a morphism of $k[z]$-algebras and extends to a morphism $\psi\colon k(z)[Z_{ij},\tfrac{1}{\det(Z)}]\to \seq_k$ of $k(z)$-algebras. As in the remark after Definition~\ref{defi:PVextension} we consider $U=k(z)[Z_{ij},\tfrac{1}{\det(Z)}]$ as $k(z)$-$\s$-algebra by virtue of $\s(Z)=AZ$. As $\s(x)=(b+1,A(b)B)$, $\s^2(x)=(b+2,A(b+1)A(b)B),\ldots$ it holds for $f=f(z,Z)\in U$ that
\begin{align*}
\psi(\s(f))&=\psi(f(z+1,A(z)Z))=\left(*,\ldots,*,f(b+1,A(b)B),f(b+2,A(b+1)A(b)B),\ldots\right) \\
 &=(*,\ldots,*,f(\s(x)),f(\s^2(x)),\ldots)=\s(\psi(f)).
\end{align*}
So $\psi\colon U\to\seq_k$ is a morphism of $k(z)$-$\s$-algebras.

By Proposition \ref{prop:existence} there exists a Picard-Vessiot ring $R\subset\seq_k$ for $\s(y)=Ay$.
Because $\psi(Z)\in\Gl_n(\seq_k)$ is a fundamental solution matrix for $\s(y)=Ay$ it follows from Remark \ref{rem:PVringunique} that $\psi(U)=R$. As explained in Section \ref{sec:PV} the ring $R$ is of the form $R=e_0R\oplus\cdots\oplus e_{l-1}R$ where
\begin{enumerate}
\item $e_j^2=e_j$ for $j=0,\ldots,l-1$,
\item $e_je_k=0$ for $j\neq k$,
\item $\s(e_0)=e_1,\s(e_1)=e_2,\ldots,\s(e_{l-1})=e_0$
\end{enumerate}
and $e_jR$ is an integral domain for $j=0,\ldots,l-1$.
The unique set ${e_0,\ldots,e_{l-1}}$ of elements of $\seq_k$ satisfying properties (i), (ii) and (iii) consists of the indicator functions
of the arithmetic progressions $j+\N l$, $j=0,\ldots,l-1$. So after possibly renumbering the $e_j$'s we can assume that
$e_j$ is the indicator function of $j+\N l$ for $j=0,\ldots,l-1$.

By assumption there exists a Picard-Vessiot extension $L|k(z)$ for $\s(y)=Ay$ inside $\seq_k$. Let $Y\in\Gl_n(L)$ be a fundamental solution matrix for $\s(y)=Ay$.
Then it follows again from Remark~\ref{rem:PVringunique} that $R=k(z)[Y_{ij},\tfrac{1}{\det(Y)}]\subset L$. So $L=k(z)(Y_{ij})$ is the total ring of fractions of $R$.

Let $g\in R$. Clearly, every non-zero element of $e_jR$ is invertible in $e_jL$. Because
$e_jL$ lives inside $\seq_k$ this implies that either $e_jg$ is zero or $e_jg$ is (eventually) non-zero on every element of $j+\N l$. In summary this shows that for every element
$g=e_0g+\cdots+e_{l-1}g\in R$ the set $\{i\in\N|\ g(i)=0\}$ is a finite union of arithmetic progressions of period $l$.


We have to show that $\{i\in\N|\ \s^i(x)\in Y\}$ is a finite union of arithmetic progressions for every closed subvariety $Y\subset X$.
The case when $Y=\V(f)$ ($f\in k[X]$) is a hypersurface follows from the above result because $\psi(f)\in R$ and
\[\{i\in\N|\ \s^i(x)\in Y\}=\{i\in\N|\ f(\s^i(x))=0\}=\{i\in\N|\ \psi(f)(i+b)=0\}.\]
The general case follows from the hypersurface case because every closed subvariety is a finite intersection of hypersurfaces and the intersection of two finite unions of arithmetic progressions is again a finite union of arithmetic progressions. This finishes the proof that (iii) implies (i).

\medskip

The implication (i) $\Rightarrow$ (ii) is immediate because the set of $b\in\N$ which are zeros of a denominator appearing in an entry of $A$ is finite.

So it only remains to prove that (ii) implies (iii). Let $x=(b,B)\in X(k)$ be such that $b\in\N$ and $\oo_\s(x)$ is defined. As in the first part of the proof we obtain a $k(z)$-$\s$-morphism $\psi\colon U\to \seq_k$. We know from Proposition \ref{prop:existence} that there exists a Picard-Vessiot ring for $\s(y)=Ay$ inside $\seq_k$ and it follows from Remark \ref{rem:PVringunique} that $\psi(U)$ is that Picard-Vessiot ring. To construct a Picard-Vessiot extension inside $\seq_k$ it suffices to show that every non-zero divisor of $\psi(U)$ is a non-zero divisor (i.e., a unit) in $\seq_k$. So let $f'\in U$ be such that $\psi(f')$ is a non-zero divisor in $\psi(U)$. Suppose for a contradiction that $\psi(f')$ is a zero divisor in $\seq_k$.

There exists a $p\in k[z]\smallsetminus\{0\}$ such that $f=pf'$ lies in $k[z,Z_{ij},\tfrac{1}{\det(Z)}]=k[X]$. Because $\psi(f')$ is a non-zero divisor in $\psi(U)$ and $\psi(p)$ is a unit in $\psi(U)$, $\psi(f)$ is a non-zero divisor in $\psi(U)$. Because $\psi(f')$ is a zero divisor in $\seq_k$, $\psi(f)$ is a zero divisor in $\seq_k$.
This means that $\psi(f)$ assumes the value zero an infinite number of times.

By assumption the set $\{i\in\N|\ \s^i(x)\in\V(f)\}=\{i\in \N |\ \psi(f)(i+b)=0\}$ is a finite union of arithmetic progressions. Thus there exists an infinite arithmetic
progression $j+\N l$ with $l\geq 2$ such that $\psi(f)$ vanishes on $j+\N l$. But then $\psi(f)\s(\psi(f))\cdots\s^{l-1}(\psi(f))=0$. Replacing $l$ with a smaller integer if necessary, we can assume that $\s(\psi(f))\cdots\s^{l-1}(\psi(f))\in\psi(U)$ is non-zero. This contradicts the fact that $\psi(f)$ is a non-zero divisor in $\psi(U)$.
\end{proof}

%

Let $k$ be a field of characteristic zero and consider a linear difference equation
\[E=\s^n(y)+h_{n-1}\s^{n-1}(y)+\cdots+h_0y=0\] over $K=k(z)$ where $\s(z)=z+1$. If $h_0=0$ then $E$ is equivalent to a linear difference equation of order strictly smaller than $n$. So we can assume without loss of generality that $h_0$ is non-zero. The equation $E$ is equivalent to
an $n\times n$ first order system $\s(y)=A(E)y$ where
\[A(E)=\left(
\begin{array}{ccccc}
0 & 1 & 0 & \cdots & 0 \\
0 & 0 & 1 & \cdots & 0 \\
\vdots & \vdots & \vdots & \vdots & \vdots \\
0 & 0 & \cdots & 0 & 1 \\
-h_0 & -h_1 & \cdots & -h_{n-2} & -h_{n-1}
\end{array}
\right).\]
An element $f$ in some $K$-$\s$-algebra is a solution of $E$ if and only if
\[\left(\begin{array}{c}
f \\
\s(f) \\
\vdots \\
\s^{n-1}(f)
  \end{array}\right)\]
is a solution of $\s(y)=A(E)y$. Note that $A(E)\in\Gl_n(K)$ because $h_0\neq 0$.

\begin{cor}\label{cor:SML1}
Let $k$ be a field of characteristic zero and
\[E=\s^n(y)+h_{n-1}\s^{n-1}(y)+\cdots+h_0y=0\] a linear difference equation over $K=k(z)$ with $h_0\neq 0$. If the equivalent conditions of Theorem \ref{theo:main} are satisfied for $A=A(E)$ then for every solution $f\in k^\N$ of $E$, the set $\{i\in\N|\ f(i)=0\}$ is a finite union of arithmetic progressions.
\end{cor}
\begin{proof}
Let $R\subset\seq_k$ denote the Picard-Vessiot ring for $\s(y)=A(E)y$. In the proof of Theorem \ref{theo:main} we have seen that for every element $g\in R$ the set $\{i\in\N |\ g(i)=0\}$ is a finite union of arithmetic progressions. Therefore it suffices to see that (the equivalence class of) $f$ lies $R$. But this follows from the simple algebraic fact that
\[\left(\begin{array}{c}
f \\
\s(f) \\
\vdots \\
\s^{n-1}(f)
  \end{array}\right)\]
must be a $k$-linear combination of the columns of a fundamental solution matrix $Y\in\Gl_n(R)$.
\end{proof}

The following result has recently been proved in \cite{Belletal:OnTheSetOfZerocoefficientsofafunctionSatisfying}. The methods and ideas used there (namely a $p$-adic analytic arc lemma and Strassman's theorem) are very similar to the methods used in \cite{Bell:AGeneralizedSkolemMahlerLechTheorem} to prove a special case of the dynamical Mordell-Lang conjecture (the case when $X$ is affine and $\s$ an everywhere defined automorphism).  Here we actually deduce the result of \cite{Belletal:OnTheSetOfZerocoefficientsofafunctionSatisfying} from \cite{Bell:AGeneralizedSkolemMahlerLechTheorem}.

\begin{cor}[{\cite[Theorem 1.2]{Belletal:OnTheSetOfZerocoefficientsofafunctionSatisfying}}] \label{cor:SML2}
Let $k$ be a field of characteristic zero and \[E=\s^n(y)+h_{n-1}(z)\s^{n-1}(y)+\cdots+h_0(z)y=0\] a linear difference equation over $k(z)$ such that
$h_{n-1},\ldots,h_1\in k[z]$ and $h_0\in k\smallsetminus\{0\}$. Then for every solution $f\in k^\N$ of $E$, the set $\{i\in\N| \ f(i)=0\}$ is a finite union of arithmetic progressions.
\end{cor}
\begin{proof}
The assumptions imply that $A(E)\in\Gl_n(k[z])$ and that the rational map \mbox{$\s\colon X\dashrightarrow X$} of Theorem \ref{theo:main} is an everywhere defined automorphism. Thus the validity of (i) (or (ii)) of Theorem \ref{theo:main} follows from \cite{Bell:AGeneralizedSkolemMahlerLechTheorem} and we can conclude via Corollary \ref{cor:SML1}.
\end{proof}

\begin{cor}
Let $k$ be a field of characteristic zero and let $A\in k[z]^{n\times n}$ be such that $\det(A)\in k\smallsetminus\{0\}$. Then there exists a Picard-Vessiot extension for $\s(y)=Ay$ inside $\seq_k$.
\end{cor}
\begin{proof}
The assumptions imply that the rational map \mbox{$\s\colon X\dashrightarrow X$} of Theorem \ref{theo:main} is an everywhere defined automorphism. Thus the validity of (i) (or (ii)) of Theorem \ref{theo:main} follows again from \cite{Bell:AGeneralizedSkolemMahlerLechTheorem}.
%
%
%
%
%
%
%
\end{proof}

In the classical case of the Skolem-Mahler-Lech theorem, i.e., for constant coefficients, there has been a vital interest in finding effective versions of the theorem, i.e., on bounding the data defining the finite union of arithmetic progressions in terms of data of the difference equation. See e.g. \cite{Evertseetal:LinearEquationsinVariables} or \cite{Amorosoetal:linearrecurrence} and the references given there. In this context it seems worthwhile to note the following fact:

\begin{prop} \label{prop:period}
Let $k$ be a field of characteristic zero and \[E=\s^n(y)+h_{n-1}(z)\s^{n-1}(y)+\cdots+h_0(z)y=0\] a linear difference equation over $k(z)$ such that
$h_0\neq 0$ and such that the equivalent statements of Theorem \ref{theo:main} are satisfied, e.g., $h_{n-1},\ldots,h_1\in k[z]$ and $h_0\in k\smallsetminus\{0\}$. Then for every solution $f\in k^\N$ of $E$, the set $\{i\in\N| \ f(i)=0\}$ is a finite union of arithmetic progressions of period less than or equal to the period of the Picard-Vessiot ring $R\subset \seq_k$ associated to $A(E)$.
\end{prop}
\begin{proof}
This is clear from Corollary \ref{cor:SML1} and the proof of the implication (iii)$\Rightarrow$(i) in Theorem \ref{theo:main}.
\end{proof}

We note that if $k$ is algebraically closed, then the Picard-Vessiot ring is unique up to $k(z)$-$\s$-isomorphisms. So the period of the Picard-Vessiot ring is an abstract algebraic invariant of the difference equation $E$, which is, a priori, not at all related to sequences.
The period of the Picard-Vessiot ring is also given by the number of connected components of the Galois group (\cite[Prop. 1.20, p. 15]{SingerPut:difference}).
Corollary 4.13 in \cite{ChatzidakisHardouinSinger:OntheDefinitionsOfDifferenceGaloisgroups} gives yet another way of computing the period of the Picard-Vessiot ring.

If $k$ is not algebraically closed one can replace the period of $R\subset \seq_k$ by the $m$-invariant of any Picard-Vessiot ring. See \cite[Prop. 4.9]{ChatzidakisHardouinSinger:OntheDefinitionsOfDifferenceGaloisgroups}.

\bigskip

The following simple example shows that the period $l$ of the Picard-Vessiot ring of $A(E)$ is not the optimal bound to write the set of zeros of any solution of $E$ as a finite union of arithmetic progressions of period $l$. However, the period of the Picard-Vessiot ring is the optimal bound to write the set of zeros of any sequence that can be obtained from solutions of $E$ by taking sums and products as a finite union of arithmetic progressions of period $l$.

\begin{ex}
We consider the Fibonacci recurrence $\s^2(y)-\s(y)-y=0$ over $\mathbb{C}(z)$. The associated matrix equation is
\begin{equation} \label{eq:fibonacci}
\s\left(\begin{array}{c} y_1 \\ y_2 \end{array}\right)=\left(\begin{array}{cc} 0 & 1 \\ 1 & 1 \end{array}\right)\left(\begin{array}{c} y_1 \\ y_2 \end{array}\right)
\end{equation}

Let $\alpha_1=\tfrac{1+\sqrt{5}}{2}$ and $\alpha_2=\tfrac{1-\sqrt{5}}{2}$ be the two solutions of the associated characteristic polynomial $t^2-t-1=0$.
Then $f_1:=(\alpha_1^n)_{n\in\N}$ and $f_2:=(\alpha_2^n)_{n\in\N}$ are two $\mathbb{C}$-linearly independent solutions in $\seq_\mathbb{C}$. The matrix
\[Y=\left(\begin{array}{cc} f_1 & f_2 \\ \s(f_1) & \s(f_2)\end{array}\right)=\left(\begin{array}{cc} f_1 & f_2 \\ \alpha_1f_1 & \alpha_2 f_2 \end{array}\right)
\]
is a fundamental solution matrix for equation (\ref{eq:fibonacci}).
It follows that $R=\mathbb{C}(z)[f_1,f_2]\subset\seq_k$ is the Picard-Vessiot ring of equation (\ref{eq:fibonacci}).
Because $f_1f_2=((-1)^n)_{n\in\N}$ there are two idempotent elements $\tfrac{f_1f_2+1}{2}=(1,0,1,\ldots)$ and $\tfrac{f_1f_2-1}{2}=(0,1,0,\ldots)$ in $R$. Consequently the period of $R$ is greater or equal to two. (It is not too hard to work out that $(f_1f_2+1)(f_1f_2-1)$ is the only algebraic relation between $f_1,f_2$ so that the period of $R$ is precisely two.)

On the other hand, any solution of the Fibonacci recurrence assumes the value zero only a finite number of times.
\end{ex}

\section{Three equivalent conjectures}

In this last section we show that the three problems SML, PV and a certain special case of DML are equivalent when considered ``globally'', i.e., without fixing the equation.

\begin{theo}
Let $k$ be a field of characteristic zero. The following statements are equivalent:
\begin{enumerate}
\item For every $A\in\Gl_n(k(z))$ the rational map $\s\colon X\dashrightarrow X$ on $X=\A^1_k\times \Gl_{n,k}$ defined by $\s(b,B)=(b+1,A(b)B)$ has the following property: For every closed subvariety $Y\subset X$ and every $x\in X(k)$ with $\oo_\s(x)$ defined the set $\{i\in\N| \ \s^i(x)\in Y\}$ is a finite union of arithmetic progressions.
\item For every solution $f\in k^\N$ of a linear difference equation
\[\s^n(y)+h_{n-1}\s^{n-1}(y)+\cdots+h_0y=0\] with $h_{n-1},\ldots,h_0\in k(z)$ the set
$\{i\in\N|\ f(i)=0\}$ is a finite union of arithmetic progressions.
\item For every $A\in\Gl_n(k(z))$ there exists a Picard-Vessiot extension of $k(z)$ for $\s(y)=Ay$ inside $\seq_k$.
\end{enumerate}
\end{theo}
\begin{proof}
The implication (i)$\Rightarrow$(iii) is immediate from Theorem \ref{theo:main}. To show (iii)$\Rightarrow$(i) let $A\in\Gl_n(k(z))$ and $x=(b,B)\in X(k)$ such that $\oo_\s(x)$ is defined. Moreover let $Y\subset X$ be a closed subvariety.


Set $\widetilde{A}(z)=A(z+b)\in\Gl_n(k(z))$ and $\widetilde{x}=(0,B)\in X$.
Let $\widetilde{\s}\colon X\dashrightarrow X$ be defined by $\widetilde{A}$ and let $\widetilde{Y}$ be the subvariety of $X$ defined by the equations $f(z+b,Z)$ where $f(z,Z)$ is a defining equation of $Y$. Then $\widetilde{\s}^i(\widetilde{x})=(i,A(b+i-1)\cdots A(b)B)$
for $i\geq 0$. In particular, $\oo_{\widetilde{\s}}(\widetilde{x})$ is defined. Moreover, $\s^i(x)=(b+i, A(b+i-1)\cdots A(b)B)$ lies in $Y$ if and only if $\widetilde{\s}^i(\widetilde{x})$ lies in $\widetilde{Y}$.
By assumption statement (iii) of Theorem \ref{theo:main} is satisfied for $\widetilde{A}$. So it follows from the implication (iii)$\Rightarrow$(i) of Theorem \ref{theo:main} applied to the $\sim$-setting that
\[\{i\in\N |\ \s^i(x)\in Y\}=\{i\in\N|\ \widetilde{\s}^i(\widetilde{x})\in\widetilde{Y}\}\]
is a finite union of arithmetic progressions.

To show (iii)$\Rightarrow$(ii) one can assume that $h_0\neq 0$. Then the claim follows from Corollary~\ref{cor:SML1}.

Finally we show that (ii) implies (iii). We already know (Proposition \ref{prop:existence}) that there exists a Picard-Vessiot ring $R$ for $\s(y)=Ay$ inside $\seq_k$. So we only have to show that every non-zero divisor of $R$ is invertible in $\seq_k$. Let $g\in R$ be a non-zero divisor in $R$ and suppose for a contradiction that $g$ is not invertible in $\seq_k$. This means that $g$ assumes the value zero an infinite number of times. As noted in Lemma \ref{lemma:sfinite} the element $g\in\seq_k$ satisfies a linear difference equation with coefficients in $k(z)$. So by (ii) there must exist an arithmetic progression $j+\N l$ with $l\geq 2$ such that $g$ vanishes on $j+\N l$. But then $g\s(g)\cdots\s^{l-1}(g)=0$. Replacing $l$ with a smaller integer if necessary, we can assume that $\s(g)\cdots\s^{l-1}(g)$ is non-zero. This contradicts the assumption that $g$ is a non-zero divisor in $R$.
\end{proof}

\bibliographystyle{alpha}
\bibliography{bibdata}
\end{document}